\documentclass[reqno]{amsproc}

\usepackage{graphicx}

\usepackage{hyperref}

\hypersetup{
colorlinks=true,
linkcolor=blue,
anchorcolor=blue,
citecolor=blue
}

\usepackage{amssymb}
\usepackage{amsfonts}
\usepackage{amsmath}
    
\newtheorem{theorem}{Theorem}[section]
\newtheorem{lemma}[theorem]{Lemma}
\newtheorem{proposition}[theorem]{Proposition}
\newtheorem{corollary}[theorem]{Corollary}

\theoremstyle{definition}
\newtheorem{definition}[theorem]{Definition}

\newtheorem{remark}[theorem]{Remark}

\numberwithin{equation}{section}

\begin{document}

\title{On a variational problem of nematic liquid crystal droplets}

\author{Qinfeng Li}
\address{School of Mathematics, Hunan University, Changsha, Hunan, China 410012}
\email{liqinfeng1989@gmail.com}
\thanks{The first author is partially supported by the National Science Fund for Youth Scholars (No. 1210010723)
and the Fundamental Research Funds for the Central Universities, Hunan Provincial Key Laboratory of
intelligent information processing and Applied Mathematics. }

\author{Changyou Wang}
\address{Department of Mathematics, Purdue University, West Lafayette, IN 47907, USA}
\email{wang2482@purdue.edu}
\thanks{The second author is partially supported by NSF grants 1764417 and 2101224.}





\begin{abstract}
Let $\mu>0$ be a fixed constant, and we prove that minimizers to the following energy functional
\begin{align*}
    E_f(u,\Omega):=\int_{\Omega}|\nabla u|^2+\mu P(\Omega)
\end{align*}exist among pairs $(\Omega,u)$ such that $\Omega$ is an $M$-uniform domain with finite perimeter and fixed volume, and $u \in H^1(\Omega,\mathbb{S}^2)$ with $u =\nu_{\Omega}$, 
the measure-theoretical outer unit normal, almost everywhere on the reduced boundary of $\Omega$. The uniqueness of optimal configurations in various settings is also obtained. 
In addition, we consider a general energy functional given by
\begin{align*}
    E_f(u,\Omega):=\int_{\Omega} |\nabla u(x)|^2 \,dx + \int_{\partial^* \Omega} f\big(u(x)\cdot \nu_{\Omega}(x)\big) \,d\mathcal{H}^2(x),
\end{align*}where $\partial^* \Omega$ is the reduced boundary of $\Omega$ and $f$ is a convex positive function on $\mathbb R$. We prove that minimizers
of $E_f$ also exist among $M$-uniform outer-minimizing domains $\Omega$ with fixed volume and $u \in H^1(\Omega,\mathbb{S}^2)$.

\end{abstract}

\maketitle


\section*{Contents}
1. Introduction

2. Prerequisite: sets of finite perimeter and traces of functions

3. {Compactness of $M$-uniform domains}

4. {Existence of equilibrium liquid crystal droplets in Problem A-C}

5. {On the uniqueness of Problem C}

References

\section{Introduction}
\label{Sec1}
In this paper we study the existence of liquid crystal droplets $(\Omega_0, u_0)$, 
consisting of a domain $\Omega_0\subset\mathbb R^3$ representing the shape of a liquid crystal drop and a unit vector field $u_0 \in H^1(\Omega, \mathbb S^2)$ representing
the average orientation field of liquid crystal molecules within the liquid crystal drop $\Omega$, that minimizes  
the total energy functional, including both the elastic energy in the bulk and the interfacial energy defined by 
\begin{eqnarray}
\label{droplet-energy}
E_f(u,\Omega):=\int_{\Omega} |\nabla u(x)|^2 \,dx + \int_{\partial^* \Omega} f\big(u(x)\cdot \nu_{\Omega}(x)\big) \,d\mathcal{H}^2(x),
\end{eqnarray}
among all pairs $(\Omega,u)$, where $\Omega$ is a domain of finite perimeter with a fixed volume that is compactly contained in the ball $B_{R_0}\subset\mathbb R^3$ with center $0$
and radius $R_0$ for some fixed constant $R_0>0$, and $u \in H^1(\Omega,\mathbb S^2)$, which is defined
by
$$H^1(\Omega,\mathbb S^2)\equiv\Big\{v\in H^1(\Omega,\mathbb R^3): \ |v(x)|=1 \ \mbox{a.e.}\ x\in\Omega\Big\}.$$ 
The functional $E_f(u,\Omega)$ should be understood in the sense that the surface integral is taken over the reduced boundary $\partial^* \Omega$ of $\Omega$, $u\ \lfloor_{\partial^* \Omega}$ is the trace of $u$ on $\partial^*\Omega$, 
$\nu_{\Omega}$ is the measure theoretical outer unit normal of $\partial^*\Omega$, and $f$ is usually assumed to have a nonnegative lower bound (with a typical choice of $f(t)=\mu(1+wt^2), t\in [-1,1],$ for some constants $\mu>0$
and $-1<w<1$). 

We will study the following minimization problem of \eqref{droplet-energy}.

\smallskip
\noindent\textbf{Problem A}. Find a pair $(\Omega, u)$ that minimizes $E_f(u,\Omega)$ over all
pairs $(\Omega,u)$ where $\Omega$ is a domain of finite perimeter
in a fixed ball $B_{R_0}\subset\mathbb R^3$, with a fixed volume $V_0>0$, and $u \in H^1(\Omega,\mathbb S^2)$,
when $f:[-1,1]\to\mathbb R$ is a nonnegative, continuous convex function.

We are also interested in the case when there is a constant contact angle condition between the liquid crystal orientation
field $u$ and the reduced boundary of liquid crystal drop $\partial^*\Omega$, i.e., $u \cdot \nu_{\Omega} \equiv c$ on $\partial^*\Omega$,  for some constant $c \in [-1, 1]$. In this case, the energy functional $E_f(u,\Omega)$ in \eqref{droplet-energy}
reduces to
\begin{eqnarray}
\label{droplet-energy1}
{E}(u,\Omega):=\int_{\Omega} |\nabla u(x)|^2 \,dx + \mu\mathcal{H}^{2}(\partial^* \Omega)
\end{eqnarray}
for some constant $\mu\ge 0$. Problem A can be reformulated as follows.

\smallskip
\noindent{\bf{Problem B}}. Find a pair $(\Omega, u)$ that minimizes ${E}(u,\Omega)$ over all
pairs $(\Omega,u)$ where $\Omega$ is a domain of finite perimeter
in a fixed ball $B_{R_0}\subset\mathbb R^3$, with a fixed volume $V_0>0$, and $u \in H^1(\Omega,\mathbb S^2)$ satisfies
$u \cdot \nu_{\Omega} \equiv c$ on $\partial^* \Omega$ for some $c\in [-1,1]$. 

We would like to mention that the contact angle condition in Problem B is referred as 
\begin{itemize}
\item [(i)] the planar anchoring condition when the constant $c=0$, and 
\item [(ii)] the homeotropic anchoring condition when the constant $c=1$.
\end{itemize}

We would like to point out that recently Geng and Lin in a very interesting paper \cite{GL} studied Problem B under the planar anchoring condition (i) in dimension two,
and proved the existence of a minimizer $(\Omega, u)$ such that  the optimal shape $\partial\Omega$ of the droplet
is a chord-arc curve with two cusps, which can be parametrized in $H^{\frac32}$ and has its unit normal vector field $\nu_\Omega$ belongs to VMO. 

Because the homeotropic anchoring condition is an important physical condition, we are also interested in
the following problem.

\smallskip
\noindent{\bf Problem C}. Find a solution to Problem B when the contact angle condition corresponds to $c=1$.  

\bigskip
\noindent\textbf{Motivation}.  The main difficulty of the minimization problems A, B, and C lies in showing
the sequential lower semicontinuity of 
$E_f(u,\Omega)$ (or ${E}(u,\Omega)$) when both domains $\Omega$ and vector fields 
$u\in H^1(\Omega,\mathbb S^2)$ vary. It is even a difficult question to ask whether
the configuration space is closed under weak convergence of liquid crystal pairs $(\Omega, u)$.

In \cite{LP},  under the assumption that 
all admissible domains $\Omega\subset B_{R_0}$ are {\it convex} domains, 
Lin and Poon have proved that there exists a minimizing pair $(\Omega_0,u_0)$ of  Problem A. 
Moreover, $u_0$ enjoys a partial regularity property similar to that of minimizing harmonic
maps by Schoen and Uhlenbeck \cite{SU1, SU2}. It was further proven by \cite{LP} that, up to translations, $(\Omega_0, u_0)=(B_{R}, \frac{x}{|x|})$ is a unique minimizer of Problem C
among convex domains with $|B_R|=V_0$.

We would like to point out that the convexity assumption of admissible domains $\Omega$ plays a crucial role in \cite{LP}, since a minimizing sequence $(\Omega_i, u_i)$ of {\it convex} domains $\Omega_i\subset B_{R_0}$ with $|\Omega_i|=V_0$ has a subsequence $\Omega_{i_k}\rightarrow\Omega$ in $L^1$, 
for some bounded convex domain $\Omega\subset B_{R_0}$ with $|B_{R_0}|=V_0$, such that $\mathcal{H}^2(\partial \Omega_{i_k})\rightarrow
\mathcal{H}^2(\partial \Omega)$ and $\nu_{\Omega_{i_k}}\rightarrow \nu_{\Omega}$ almost everywhere with respect to
a spherical coordinate system\footnote{For example, one can parametrize $\partial\Omega_{i_k}$ and
$\partial\Omega$ over the unit sphere $\mathbb S^2$.}. Moreover, there exists $u\in H^1(\Omega,\mathbb S^2)$
such that $\nabla u_{i_k}\chi_{\Omega_{i_k}}\rightarrow \nabla u \chi_{\Omega}$ weakly in $L^2(\mathbb R^3)$. 
The uniqueness of minimizer of Problem C among convex domains relies on the following important inequalities:\begin{eqnarray}
\label{xin1}
\int_{\Omega} |\nabla u(x)|^2 \,dx \ge \int_{\partial \Omega} H(x) \,d\mathcal{H}^2(x), \ \forall u \in H^1(\Omega, \mathbb S^2) \ {\rm{with}}\  u=\nu_{\Omega}
\ {\rm{a.e.\ on}}\ \partial \Omega,
\end{eqnarray}
and  \begin{eqnarray}
\label{xin2}
\int_{\partial \Omega} H(x) \,d\mathcal{H}^2(x) \ge \sqrt{ 4\pi \mathcal{H}^2 (\partial \Omega})\ \mbox{for
convex $\Omega$}, \  {\rm{equality\ holds\ iff}}\ \Omega =B_R,
\end{eqnarray} 
where $H$ denotes the mean curvature of $\partial \Omega$. In \cite{LP}, \eqref{xin1} is derived for any $\Omega\in W^{2,1}$, while \eqref{xin2} is proven by the Brunn-Minkowski inequality for convex domains. 

In this paper, we would like to relax the convexity assumption from \cite{LP} and investigate
Problems A, B, and C  over a larger class of domains  possibly containing {\it non-convex} domains with {\it less regular} boundaries. The class of domains 
contains Sobolev extension domains with some uniform parameters, as well outer minimal domains. 

The main theorems of this paper arose from the Ph.D. thesis of the first author \cite{Li}. The interested reader can refer to \cite{Li} for more related results.

\medskip
\noindent\textbf{Outline of this paper}:\\
\indent In section 2, we will review certain classes of domains in $\mathbb{R}^n$, including $M$-uniform domains, which are Sobolev extension domains with constants depending on $M$ and $n$;
the outer minimal domains, which are a generalization of convex domains.

In section 3, we will show in Theorem \ref{maincompact} that, up to a set of measure zero, the $L^1$-limit of $M$-uniform domains is $M$-uniform. A few other results on the relation between $L^1$-convergence and Hausdorff convergence are also derived. 

In section 4, we will establish the weak lower semicontinuity of bulk elastic energy of $(\Omega,u)$ for two classes of domains: a) the admissible sets of $M$-uniform domains,
and b) the admissible sets of outer minimal $M$-uniform domains. It is more subtle to prove the lower semicontinuity of surface energy for Problem A. We will only consider outer minimal sets and our proof is inspired by Reshetnyak's lower semicontinuity theorem (see \cite[Theorem 20.11]{Maggi})
and the perimeter convergence Lemma \ref{heng}. Thus combining the compactness of $M$-uniform domain results and the lower-semicontinuity results, the existence Theorem \ref{bigexistence} on problems A, B and C is proved among these classes of admissible sets.

In section 5, we will apply results by \cite{DHMT}, \cite{FS}, \cite{HI1} and \cite{HI2} to show $(B_R, \frac{x}{|x|})$ is the unique minimizer of Problem C over strictly star-shaped mean convex $C^{1,1}$-domains, $C^{1,1}$-outer minimal sets,
and $C^{1,1}$-revolutionary domains, see Theorem \ref{dafeiji1} and Remark \ref{dafeiji2}. 

\section{Prerequisite: sets of finite perimeter and traces of functions}

We first stipulate some notations. Let $V_0>0$ be the fixed volume in the Problems A, B, C. Since the admissible domains in the problems have this fixed volume, we will use the convention that any minimizing sequences have their diameters larger than a universal constant $c_0=c_0(V_0)>0$ because of the isodiametric inequality (see \cite[Theorem 2.2.1]{EG}).

We will denote by $B_r(x): =\{y \in \mathbb{R}^n: |y-x|<r\}$ and $B_r:=B_r(0)$. Throughout this paper all sets under consideration are contained in a large ball $B_{R_0}$, where $R_0>0$ is fixed.  
For any set $A\subset\mathbb R^n$, denote by $A_{\epsilon}$  the interior $\epsilon$-neighborhood $\{x \in A: B_{\epsilon}(x) \subset A\}$, and $A^{\epsilon}$ the exterior $\epsilon$-neighborhood $\bigcup_{x \in A} B_{\epsilon}(x)$. 
Denote by ${\rm{int}} (A)$  the topological interior part of $A$, $A^c=\mathbb{R}^n \setminus A$, and ${\rm{diam}}(A)$  the diameter of $A$. 
For $0\le d\le n$, $\mathcal{H}^d$ denotes the $d$-dimensional Hausdorff measure in $\mathbb R^n$. Let $d^H(\cdot,\cdot)$ 
denote the Hausdorff distance in $\mathbb R^n$. $P(A;D)$ denotes the distributional perimeter of $A$ in $D\subset\mathbb R^n$. 
For a set $A$ of finite perimeter, let $\nu_A$ denote the measure theoretical outer unit normal of the reduced boundary $\partial^*A$,
and $\mu_A$ denotes the Gauss-Green measure of $A$, that is, $\mu_A=\nu_A \cdot \mathcal{H}^{n-1}\lfloor_{\partial^*A}$.
Denote by $\omega_n$ the volume of unit ball in $\mathbb{R}^n$ and $|A|$ the Lebesgue measure of $A$.. 
For any open set $\Omega\subset\mathbb R^n$ and $u \in BV(\Omega)$, denote by $Du$ the  
distributional derivative of $u$, that is a vector-valued Radon measure, and $\|Du\|(\Omega)$ the total variation of $u$ on $\Omega$.

In this paper, $``\lesssim_c"$ denotes an inequality up to constant multiplier $c>0$.
For any measurable set $E$ and $0\le\alpha\le1$, we define
$$E^\alpha=\Big\{x \in \mathbb{R}^n: \lim_{r\rightarrow0} \frac{|E\cap B_r(x)|}{|B_r(x)|}=\alpha\Big\},$$
and refer $E^1$ and $E^0$  as the measure theoretical interior and exterior part of $E$ respectively.
Denote by $\partial_*E:=\mathbb{R}^n
\setminus(E^0 \cup E^1)$ the measure theoretical boundary of $E$, which is also called the essential boundary. In this paper, we will need the following theorem, due to Federer (see \cite[Chapter 5]{EG}).

\begin{theorem}
\label{Federer}
For any measurable set $E$, if $\mathcal{H}^{n-1}(\partial_* E)< \infty$, then $E$ is a set of finie perimeter. Furthermore, if $E$ is a set of finite perimeter, then $\mathbb{R}^n=E^0 \cup E^1 \cup \partial_* E$, $\partial^* E \subset E^{(1/2)} \subset \partial_*E$, and $\partial^* E =\partial_* E \, ({\rm{mod}}\ \mathcal{H}^{n-1})$.
\end{theorem} 

Next, we recall the definition of $M$-uniform domains.
\begin{definition}
\label{uniformdomain}
For $M\ge1$, a domain $\Omega\subset\mathbb R^n$ is called an $M$-uniform domain,
if for any  two points $x,y \in \Omega$, there is a rectifiable curve $\gamma : [0,1] \rightarrow \Omega$
such that $\gamma(0)=x,\ \gamma(1)=y$, and
\begin{eqnarray}
\label{Jones1}
 &\,& \mathcal{H}^1(\gamma([0,1])) \le M|x-y|, \\
&\,&  d(\gamma(t), \partial \Omega) \ge \frac{1}{M} \min\big\{|\gamma(t)-x|,|\gamma(t)-y|\big\}, \, \forall t \in [0,1].
\label{Jones2}
\end{eqnarray}

\end{definition}

\begin{remark} {\rm P. Jones \cite{Jo} introduced the notion of $(\epsilon, \delta)$-domain. One can check
that  any $(\epsilon,\infty)$-domain is an $M$-domain, with $M=\frac{2}{\epsilon}$. On the other hand, any $M$-uniform
domain is a $(\frac{1}{M^2},\infty)$-domain\footnote{Since \eqref{Jones1} and \eqref{Jones2} imply
$$d(\gamma(t),\partial\Omega)\ge \frac{1}{M}\frac{|\gamma(t)-x||\gamma(t)-y|}{\mathcal H^1(\gamma([0,1]))}
\ge \frac{1}{M^2}\frac{|\gamma(t)-x||\gamma(t)-y|}{|x-y|},  \ \forall t\in [0,1].$$}.
It was also proven by \cite{Jo} that any $(\epsilon,\delta)$ domain is a Sobolev extension domain, 
and the converse is true when $n=2$.  We refer to \cite{GO} and \cite{Jo} for more details on $M$-uniform domains.}
\end{remark}

Since we will study minimization problems involving traces of bounded $H^1$ vector fields in this paper, we will need the following Gauss-Green formula.

\begin{theorem}
\label{traceextensiondomain}
Let $\Omega$ be a bounded uniform domain of finite perimeter in $\mathbb{R}^n$ and $u \in H^1(\Omega)\cap L^{\infty}(\Omega)$. Then for any $\phi \in C_0^1(\mathbb{R}^n,\mathbb{R}^n)$, we have 
\begin{eqnarray}
\label{fenbujifen}
\int_{\Omega} u {\rm{div}}\phi +\int_{\Omega} \phi Du =\int_{\partial^* \Omega} (\phi\cdot\nu_{\Omega})u^* d\mathcal{H}^{n-1},
\end{eqnarray}
where  $\nu_{\Omega}$ is the measure-theoretic unit outer normal to $\partial^* \Omega$, and $u^*$ is given by the formula 
\begin{eqnarray}
\label{GGlip}
\lim_{r \rightarrow 0}\displaystyle \frac{\int_{B_r(x) \cap \Omega}|u-u^*(x)|}{r^n}=0,\  \mbox{$\mathcal{H}^{n-1}$-a.e. $x\in \partial^* \Omega$}.
\end{eqnarray}
\end{theorem}

\begin{proof}
According to \cite{Jo}, we may let $\hat{u} \in H^1_0(\mathbb{R}^n)\cap L^{\infty}(\mathbb{R}^n)$ be an extension of $u$ such that $\hat{u}=u$ in $\Omega$ and
\begin{align*}
    \Vert \hat{u} \Vert _{H^1(\mathbb{R}^n)} \le C(n,\Omega) \Vert u \Vert _{H^1(\Omega)}.
\end{align*}
Hence $\hat{u} \in BV(\mathbb{R}^n)$, and thus according to \cite[Theorem 3.77]{afp}, the interior trace of $\hat{u}$, denoted by $\hat{u}^*$ here, is well-defined 
for $\mathcal{H}^{n-1}$-a.e. on $\partial^*\Omega$, and equals to $u^*$,  given by \eqref{GGlip}, for $\mathcal{H}^{n-1}$-a.e. on $\partial^* \Omega$. Let $\tilde{u}=\hat{u}\chi_\Omega$. Since $\hat{u}$ is bounded, $u^* \in L^1(\partial^* \Omega)$ and thus by \cite[Theorem 3.84]{afp}, $\tilde{u}=\hat{u}\chi_{\Omega} \in BV(\mathbb{R}^n)$, with
\begin{align*}
    D\tilde{u}=D\hat{u}\lfloor_{\Omega^1}-u^*\nu_\Omega\mathcal{H}^{n-1}\lfloor_{\partial^*\Omega}.
\end{align*} Hence for any $\phi \in C_0^1(\mathbb{R}^n,\mathbb{R}^n)$, we have
\begin{align}
\label{ji1}
    \int_{\mathbb{R}^n}\phi D\tilde{u}=\int_{\Omega^1} \phi D\hat{u}-\int_{\partial^* \Omega}(\phi \cdot \nu_{\Omega})u^*\, d\mathcal{H}^{n-1}. 
\end{align}Since
\begin{align*}
    \int_{\mathbb{R}^n}\phi D\tilde{u}=-\int_{\mathbb{R}^n}\tilde{u}\div \phi=-\int_{\Omega}u\div \phi,
\end{align*}from \eqref{ji1} we have
\begin{align}
\label{j2}
     \int_{\Omega}u\div \phi+\int_{\Omega^1}\phi D\hat{u}=\int_{\partial^*\Omega}(\phi\cdot \nu_{\Omega})u^*d\mathcal{H}^{n-1}.
\end{align}
Since $\Omega$ is equivalent to $\Omega^1$ up to a set of Lebesgue measure zero and $\hat{u}\in H^1(\mathbb{R}^n)$, we have
\begin{align}
\label{ji3}
    D\hat{u}\lfloor_{\Omega^1}=D\hat{u}\lfloor_\Omega=Du\lfloor_{\Omega}
\end{align}
Hence \eqref{j2} and \eqref{ji3} imply \eqref{fenbujifen}.
\end{proof}

For the purpose later in this paper, we also introduce the following definition. 

\begin{definition} 
\label{D-c}
For any $c>0$, we denote by $\mathcal{D}_c$ the class of bounded sets in $\mathbb R^n$ such that for any set $E \in \mathcal{D}_c$, 
\begin{eqnarray}
\label{g1}
|B_r(x)\cap E|>cr^n
\end{eqnarray} holds for any $x \in \partial E$ and $0<r<{\rm{diam}}(E)$. 
\end{definition}

Recall that two sets $E, F\subset \mathbb R^n$ are said to be $\mathcal{H}^n$-equivalent, denoted by $E\approx F$,
if $E\Delta F=(E\setminus F)\cup (F\setminus E)$ has zero Lebesgue measure. 
Note that by the Lebesgue density theorem, if $E \in \mathcal{D}_c$, then $|\partial E \cap E^c|=0$. Hence
$\partial E\subset E$ (mod $\mathcal{H}^{n}$) and $\overline{E}\approx E$. In particular, we have

\begin{remark}\label{closure} {\rm Any  $E\in \mathcal{D}_c$  is equivalent to its closure $\overline{E}$.}
\end{remark}

We also have

\begin{remark}
\label{feihua1}{\rm
For $c>0$, if $E\in \mathcal{D}_c$ is a set of finite perimeter, then 
there is $c'>0$ depending only on $c$ and $n$ such that 
for any $x \in \overline{E}$ and $0<r<{\rm{diam}}(E)$, $|B_r(x) \cap E| \ge c'r^n$. } 
\end{remark}

\begin{proof}
For $x \in \overline{E}$ and $0<r<{\rm{diam}}(E)$, there are two cases: \\
(a) If $r \ge 2 d(x, \partial E)$, then there is $z \in \partial E$ such that $B_{\frac{r}{2}}(z) \subset B_{r}(x)$. 
Hence 
$$|B_{r}(x) \cap E| \ge |B_{\frac{r}2}(z) \cap E| \ge c(\frac{r}{2})^n=\frac{c}{2^n} r^n.$$
(b) If $r \le 2 d(x,\partial E)$, then $B_{\frac{r}{2}}(x) \subset E$ and hence
$$|B_{r}(x) \cap E| \ge |B_{\frac{r}2}(x)|=\frac{\omega_n}{2^n}r^n.$$
Hence the conclusion holds with $c'=\min\{\frac{c}{2^n}, \frac{\omega_n}{2^n}\}$.
\end{proof}

The following proposition shows that any $M$-uniform domain belongs to $\mathcal{D}_c$ for some $c>0$.
\begin{proposition}
\label{shuyu} For any $M\ge 1$ and $c_0>0$, if $\Omega\subset\mathbb R^n$ is an $M$-uniform domain, 
with ${\rm{diam}} (\Omega) \ge c_0>0$, then $\Omega \in \mathcal{D}_c$ for some $c>0$ depending only on $M$, $n$ and $c_0$.
\end{proposition}

\begin{proof}
For any $x \in \partial \Omega$ and $0<r<{\rm{diam}}(\Omega)$, we claim 
that there is a constant $c_1=c_1(M)>0$ such that $B_r(x) \cap \Omega$ contains a ball of radius $c_1r$. 
Indeed, since $0<r<{\rm{diam}}(\Omega)$, there is $y \in \Omega \setminus B_{\frac{r}2}(x)$. Let $\gamma$ be the curve joining $x$ and $y$ given by the definition of $M$-uniform domain. Choose $z \in \partial B_{\frac{r}{3}}(x) \cap \gamma$.
Then we have that $z \in \Omega$ and 
$$d(z, \partial \Omega) \ge \frac{1}{M}\min\big\{|z-x|, |z-y|\big\}\ge \frac{1}{M}\min\big\{\frac{r}{3},
\frac{r}2-\frac{r}3\big\}=\frac{r}{6M}.$$ 
Hence $B_{c_1r}(z) \subset \Omega$, with $c_1=\frac{1}{6M}$. 
From this claim, we see that for any $x \in \partial \Omega$ and any $r<{\rm{diam}}(\Omega)$, 
$$|B_r(x) \cap \Omega| \ge |B_{c_1 r}(z)|\ge \omega_nc_1^nr^n.$$ 
This completes the proof.
\end{proof}

The following remark will be used in the proof of compactness of $M$-uniform domains.
\begin{remark}
\label{bukong}{\rm For $M>0$ and $c_0>0$, 
if $\Omega\subset\mathbb R^n$ is an $M$-uniform domain, with $|\Omega| \ge c_0$, 
then there is $r_0>0$ depending only on $M,n,c_0$ such that $\Omega$ contains a ball of radius $r_0$. }
\end{remark}

\begin{proof} It follows directly from the isodiametric inequality
and Proposition \ref{shuyu}.  
\end{proof}

Similar to $\mathcal D_c$, we also define the class $\mathcal{D}^c$ as follows.
\begin{definition}
\label{D^c}
For $c>0$, the set class $\mathcal{D}^c$  consists of all bounded set $E\subset\mathbb R^n$ such that  
\begin{eqnarray}
\label{g2}
|B_r(x)\cap E^c|>cr^n
\end{eqnarray} holds for any $x \in \partial E$ and $0<r<{\rm{diam}}(E)$. 
\end{definition}

The following proposition from \cite[Proposition 12.19]{Maggi} yields that we can always find an $\mathcal{H}^n$-equivalent set $\widetilde{E}$ of any set $E$ of finite perimeter 
with slightly better topological boundary.
\begin{proposition}
\label{equiv}
For any Borel set $E\subset\mathbb R^n$, there exists an $\mathcal{H}^n$-equivalent set $\widetilde{E}$ of $E$
such that for any $x \in \partial \widetilde{E}$ and any $r>0$, 
\begin{eqnarray}
\label{spt}
0<|\widetilde{E} \cap B_r(x)|<\omega_nr^n.
\end{eqnarray}
In particular, ${\rm{spt}} \mu_E={\rm{spt}}\mu_{\widetilde{E}}=\partial \widetilde{E}$.
\end{proposition}

In order to illustrate the construction of such an equivalent set, which is needed in later sections, we will sketch the proof.
\begin{proof} 
First, we define two disjoint open sets 
$$A_1:=\big\{x \in \mathbb{R}^n\ |\ \mbox{there exists $r>0$ such that $|E \cap B_r(x)|=0$}\big\},$$ 
and 
$$A_2:=\big\{x \in \mathbb{R}^n\ |\ \mbox{there exists $r>0$ such that $|E \cap B_r(x)|=\omega_nr^n$}\big\}.$$ 
Then by simple covering arguments we have
that $|E \cap A_1|=0$ and $|A_2 \setminus E|=0$. 
Set $\widetilde{E}=(A_2 \cup E) \setminus A_1$. Then 
$$|\widetilde{E} \Delta E| \le |A_2 \setminus E|+|E \cap A_1|=0.$$ 
Moreover, since $A_2 \subset {\rm{int}}(\widetilde{E})$ and  $\overline{\widetilde{E}} \subset \mathbb{R}^n \setminus A_1$, 
we have that $\partial \widetilde{E} \subset \mathbb{R}^n \setminus (A_1 \cup A_2)$ and hence \eqref{spt} holds.
\end{proof}

We now recall the notion of outer minimal sets, which can be viewed as a subsolution of 
area minimizing sets. It is a generalization of convex sets, see for example \cite[Definition 15.6]{Giusti} and related results therein.

\begin{definition}
\label{psc}
A set $E\subset\mathbb R^n$ of finite perimeter is an outer minimal set, 
if $P(E) \le P(F)$ holds for any set $F \supset E$. 
\end{definition}

We would like to point out that an outer-minimal set is also called as a pseudo-convex set by \cite{LT}. Thus 
by \cite[Corollary 7.16]{LT} we have
\begin{remark}
\label{dens}
{\rm If $E\subset\mathbb R^n$ is an outer-minimizing and ${\rm{spt}} \mu_E=\partial E$, then $E \in \mathcal{D}^c$, for some $c>0$ depending only on $n$ and $E$. Consequently, $E={\rm{int}}(E)\ ({\rm{mod}}\mathcal{H}^n)$.}
\end{remark}


\begin{remark} {\rm Since the boundary of an outer minimal set (domain) can have positive $\mathcal{H}^n$ measure (see \cite{BGM}),  an outer minimal domain may not be an $M$-uniform domain for any $M\ge 1$.}
\end{remark}

Combining Proposition \ref{shuyu} and Remark \ref{dens}, we have
\begin{remark}
\label{bianjiea}
Let $\Omega$ be an $M$-uniform outer minimal domain with $\rm{spt} \mu_\Omega=\partial \Omega$, then $\Omega \in \mathcal{D}_c\cap \mathcal{D}^c$ for some $c>0$, and hence $\partial_*\Omega=\partial \Omega$.
\end{remark}

We would like to state the following proposition, which is a consequence of \cite[Corollary 1.10]{GHL}, since for any $E \in \mathcal{D}_c$, $\mathcal{H}^{n-1}(\partial E\cap E^0)=0$. 
\begin{proposition}
\label{outside}
Let $c>0$ and $E \in \mathcal{D}_c$. Then there exists bounded smooth sets $E_i$ such that $E_i \Supset E$, $E_i \rightarrow E$ in $L^1$ and $P(E_i)\rightarrow P(E_i)$.
\end{proposition}

\section{Compactness of $M$-uniform domains}
In this section, we will establish in Theorem \ref{maincompact} the
$L^1$-compactness property of $M$-uniform domains. We begin with
\begin{lemma}
\label{neijin} For $c>0$, 
suppose that $\{D_i\}\subset\mathcal{D}_c$ satisfies $D_i \rightarrow D$ in $L^1(\mathbb R^n)$
as $i\rightarrow\infty$. 
Then after modifying over a set of Lebesgue measure zero, $D \in \mathcal{D}_c$. Moreover, for any $\epsilon>0$, there is $N=N(\epsilon)>0$ such that for any $i>N$, the following properties hold:\\
(i) $D \subset D_i^{\epsilon}$.\\
(ii) $(D_i)_{\epsilon} \subset D$.\\
(iii) $D_i \subset D^{\epsilon}$. \\
In particular, $d^H(D_i,D) \rightarrow 0$ as $i\rightarrow\infty$. 

\end{lemma}

\begin{proof} 
We first identify $D$ with its $\mathcal{H}^n$-equivalent set in the sense of Proposition \ref{equiv}. We argue by contradiction.

If (i) were false, then there would exist $\epsilon_0>0$, $x_0 \in D$ and a sequence $k\rightarrow\infty$
such that $B_{\epsilon_0}(x_0) \cap D_k=\emptyset$.  Hence by the hypothesis and Proposition \ref{equiv}, we obtain 
that 
$$0=|B_{\epsilon}(x_0) \cap D_k|\rightarrow |B_{\epsilon}(x_0) \cap D|>0,$$
this is impossible.

If (ii) were false, then there would exist $\epsilon_0>0$ and a sequence of points $x_i \in (D_i)_{\epsilon_0} \setminus D$. 
Assume that $x_i \rightarrow x_0$. Then $x_0 \in \partial D \cup D^c$. Hence by the proof of Proposition \ref{equiv}, we have
that $\omega_n\epsilon_0^n>|B_{\epsilon_0}(x_0) \cap D|$. On the other hand, since $B_{\epsilon_0}(x_i)\subset D_i$, we have that
\begin{eqnarray*}
\big|B_{\epsilon_0}(x_0) \cap D\big|&=&\lim_{i \rightarrow \infty}\big|B_{\epsilon_0}(x_i) \cap D\big|
\ge\liminf_{i \rightarrow \infty} \big(|B_{\epsilon_0}(x_i) \cap D_i|-|D_i \Delta D|\big)\\
&=&\omega_n\epsilon_0^n-\limsup_{i \rightarrow \infty}|D_i \Delta D|=\omega_n\epsilon_0^n.
\end{eqnarray*}
We get a desired contradiction.

If (iii) were false, then there would exist $\epsilon_0>0$ and a subsequence of $x_i \in D_i\setminus D^{\epsilon_0}$. Without loss of generality, assume $x_i \rightarrow x_0$ and thus $x_0 \in \mathbb{R}^n \setminus D^{\epsilon_0}$. By Remark \ref{feihua1}, there is a $c'>0$ depending only on $c$ and $n$ such that 
$$c'\epsilon_0 ^n \le \big|B_{\epsilon_0}(x_i) \cap D_i\big|.$$
On the other hand, it follows from $|B_{\epsilon_0}(x_0) \cap D|=0$
that
\begin{eqnarray*}
\liminf_{i \rightarrow \infty}\big|B_{\epsilon_0}(x_i) \cap D_i\big| 
&\le& \limsup_{i\rightarrow \infty}\big(|B_{\epsilon_0}(x_i) \cap D|+|D \Delta D_i|\big) \\
&\le& |B_{\epsilon_0}(x_0) \cap D|+\limsup_{i \rightarrow \infty} |D_i \Delta D|=0.
\end{eqnarray*}
This yields a desired contradiction.

It remains to show $D \in \mathcal{D}_c$. Indeed, by Proposition \ref{equiv}, $x \in \partial D$ implies
that $x \in {\rm{spt}}\mu_D$. 
Note $D_i \rightarrow D$ in $L^1(\mathbb R^n)$ implies that
$\mu_{D_i} \stackrel{*}{\rightharpoonup} \mu_D$ as convergence of Radon measures. Hence there exists 
$x_i \in {\rm{spt}}\mu_{D_i} \subset \partial D_i$ such that $x_i \rightarrow x$ so that for any $r>0$, it holds
that
$$\big|B_r(x) \cap D\big|=\lim_i \big|B_r(x_i) \cap D\big| 
\ge \liminf_i \big|B_r(x_i) \cap D_i\big|-\limsup_i\big|D_i \Delta D\big| 
\ge cr^n.$$ 
This implies $D \in \mathcal{D}_c$.  
\end{proof}

The following remark follows directly from  (i) and (iii).

\begin{remark}
\label{new2}{\rm 
If $D_i$ and $D$ satisfy the same assumptions as in Lemma \ref{neijin}, and if ${\rm{int}}(D) \ne \emptyset$, 
then ${\rm{int}}(D)$ is connected.}
\end{remark}


Similar to Lemma \ref{neijin},  for a set in the class $\mathcal{D}^c$ we have 
\begin{lemma}
\label{waijin} For $c>0$, if $\{D_i\} \subset \mathcal{D}^c$ and $D_i \rightarrow D$ in $L^1(\mathbb R^n)$, then
after modifying a set of zero $\mathcal{H}^n$-measure, $D \in \mathcal{D}^c$. Moreover, for any $\epsilon>0$, there is  $N=N(\epsilon)>0$ such that if $i>N$, the following properties holds:\\
(i) $D \subset D_i^{\epsilon}$.\\
(ii) $(D_i)_{\epsilon} \subset D$.\\
(iii) $D_{\epsilon} \subset D_i$. 
\end{lemma}
The following corollary follows directly from Lemma \ref{waijin}.
\begin{corollary}
\label{new1} For any $c>0$ and a sequence $\{D_i\}\subset\mathcal{D}^c$ with uniformly bounded perimeters, 
there is an open set $D \in \mathcal{D}^c$ such that $D_i \rightarrow D$ in $L^1(\mathbb R^n)$. Moreover, $D$ and $D_i$ satisfy the properties {\rm{(}}i{\rm{)}},{\rm{(}}ii{\rm{)}} and {\rm{(}}iii{\rm{)}} of Lemma \ref{waijin}.
\end{corollary}


Now we are ready to prove the main theorem of this section.

\begin{theorem}
\label{maincompact} For $M>0$,  $R_0>0$, and $c_0>0$,
if $\{\Omega_i\}$ is a sequence of $M$-uniform domains in $B_{R_0}$
such that $|\Omega_i| \ge c_0>0$ and $\Omega_i \rightarrow D$ in $L^1(\mathbb R^n)$, 
then there is an $M$-uniform domain $\Omega$ such that $\Omega_i \rightarrow \Omega$ in $L^1(\mathbb R^n)$.
\end{theorem}

\begin{proof} As in Proposition \ref{equiv}, we assume ${\rm{spt}} \mu_D=\partial D$. 
We first prove that ${\rm{int}}(D) \ne \emptyset$. Indeed, notice that by Remark \ref{bukong}, 
there exists a $r_0>0$ depending only on $c_0,n$ and $M$ such that each $\Omega_i$ contains a ball of radius $r_0$. Therefore, for each $\Omega_i$, if $\epsilon<\frac{r_0}2$, then by definition $(\Omega_i)_{\epsilon}$ contains a ball of radius $\frac{r_0}2$. By Lemma \ref{neijin} (ii), $D$ also contains a ball of radius $\frac{r_0}2$ and hence
${\rm{int}}(D) \ne \emptyset$.

Set $\Omega={\rm{int}}(D)$. It suffices to show that $\Omega$ is an $M$-uniform domain, since the $L^1$ convergence 
of $\Omega_i$ to $\Omega$ follows directly from Remark \ref{closure}, Proposition \ref{shuyu},
and the fact $\Omega \subset D \subset \overline{\Omega}$.

Fix any $x,y \in \Omega$, then given any $N>>M$, say $N>2M$, we may choose $0<\epsilon<\frac{1}{N}$ so small that $k\epsilon < d(x, \partial \Omega) \le (k+1)\epsilon$, $k>>N$ (say $k>(1+1/M)(N+1)$), and $|x-y|>2(N+1)\epsilon$. From Lemma \ref{neijin} (i) and (iii), and since $int(\Omega) \ne \emptyset$, we know that $d^H(\Omega_i, \Omega) \rightarrow 0$, hence we we may choose $x_i, y_i \in \Omega_i \cap \Omega$, with $|x_i-x|<\epsilon, |y_i-y|<\epsilon$ for $i$ large. By Lemma \ref{neijin} (ii), we may also choose $i$ large such that 
\begin{eqnarray}
\label{cru}
(\Omega_i)_{\epsilon} \subset \Omega.
\end{eqnarray} Also we choose $\gamma_i \subset \Omega_i$ to be the rectifiable curve connecting $x_i$ and $y_i$ in $\Omega_i$ as in the definition of $M$-uniform domain. For any $p \in \gamma_i$, if $p \in B_{N \epsilon}(x_i) \cup B_{N\epsilon}(y_i)$, then clearly $p \in B_{(N+1) \epsilon}(x) \cup B_{(N+1)\epsilon}(y) \subset \Omega$. Moreover, this implies
\begin{equation}
\label{a1}
d(p, \partial \Omega) \ge k\epsilon-(N+1) \epsilon>\frac{1}{M}(N+1)\epsilon\ge \frac{1}{M}\min\{|p-x|,|p-y|\}.
\end{equation}
Clearly \eqref{a1} also holds for any $p$ on the line segment between $x_i$ and $x$, and between $y_i$ and $y$.

If $p \notin B_{N \epsilon}(x_i) \cup B_{N\epsilon}(y_i)$, then $d(p, \partial \Omega_i) \ge \frac{1}{M} \min\{|p-x_i|,|p-y_i|\}> \frac{1}{M} N\epsilon$, thus $p \in (\Omega_i)_{N\epsilon/M} \subset (\Omega_i)_{\epsilon}\subset \Omega \cap \Omega_i$. Moreover, let $r=d(p,\partial((\Omega_i)_{\epsilon}))$, then by \eqref{cru}, $B_r(p) \subset \Omega$, so $d(p,\partial \Omega) \ge r=d\left(p, \partial ((\Omega_i)_{\epsilon})\right) \ge d(p,\partial \Omega_i)-\epsilon$. Therefore,
\begin{equation}
\label{a3}
\frac{d(p,\partial \Omega)}{\min\{|p-x_i|,|p-y_i|\}} \ge \frac{d(p,\partial \Omega_i)-\epsilon}{\min\{|p-x_i|,|p-y_i|\}}\ge \frac{1}{M}-\frac{\epsilon}{N\epsilon}\ge \frac{1}{M}-\frac{1}{N}.\\
\end{equation}
Hence by the choice of $\epsilon$ and $N$ we have that
\begin{equation}
\label{a2}
d(p,\partial \Omega) \ge (\frac{1}{M}-\frac{1}{N})(\min\{|p-x|,|p-y|\}-\epsilon) \ge(\frac{1}{M}-\frac{1}{N})(\min\{|p-x|,|p-y|\})-\frac{1}{MN}.
\end{equation}

Therefore, we may let $\gamma^N$ be the curve with three parts. The first part connects $x$ and $x_i$ with line segment, the second part connects $x_i$ and $y_i$ with $\gamma_i$ as above and the third part connects $y_i$ and $y$ with line segment. It is clear that $\gamma^N \subset \Omega$ and $\gamma^N$ connects $x$ and $y$, then from \eqref{a1} and \eqref{a2} and the choice of $\epsilon$, we obtain\\
(i) $\mathcal{H}^1(\gamma^N) \le M|x-y|+2\frac{M+1}{N}$, and\\
(ii) $d(p, \partial \Omega) \ge (\frac{1}{M}-\frac{1}{N}) \min\{|p-x|,|p-y|\}-\frac{1}{MN}\quad \forall p\in \gamma^N$. \\
Then by compactness of $(\overline{\Omega}, d^H)$, and since $\gamma^N$ is connected, there is a compact connected set $E \subset \overline{\Omega}$ such that $d^H(\gamma^N, E) \rightarrow 0$ as $N \rightarrow \infty$. Then by \cite[Theorem 3.18]{Fa}, 
$$\mathcal{H}^1(E) \le \liminf_{N \rightarrow \infty} \mathcal{H}^1(\gamma^N) \le M|x-y|.$$
Then by \cite{Fa}[Lemma 3.12], $E$ is path connected, thus we can choose a curve $\gamma \subset E$ joining $x$ and $y$. For any $p \in \gamma$, we can choose sequence $p_N\in \gamma^N, p_N \rightarrow p$. Since 
$$d(p_N, \partial \Omega) \ge (\frac{1}{M}-\frac{1}{N}) \min\{|p_N-x|,|p_N-y|\}-\frac{1}{2MN},$$ 
we have, after sending $N \rightarrow \infty$,
$$d(p, \partial \Omega) \ge \frac{1}{M} \min\{|p-x|,|p-y|\},$$
which also clearly implies $\gamma \subset int\,\Omega$. Then $\gamma$ satisfies both properties in the definition of $M$-uniform domain, thus $\Omega$ is $M$-uniform.
By Remark \ref{new2} and Proposition \ref{shuyu}, $\Omega$ is a domain. This completes the proof.
\end{proof}

\begin{remark}
The full generality of compactness of $M$-uniform domains is obtained in \cite[Theorem 1.2]{DLW}, where it is shown that any sequence of $M$-uniform domains with fixed volume must have uniformly bounded fractional perimeters, and thus have an $L^1$ limit up to a subsequence, and the limit is also $M$-uniform.
\end{remark}

\section{Existence of equilibrium liquid crystal droplets in Problem A-C}
In this section we will study the existence of minimizers to Problems A-C, which can be extended in $n$-dimensions. 
We begin with the following Lemma, which plays a crucial role in Problems A-C over outer minimal sets.

\begin{lemma}
\label{heng}
For $c>0$, let $\{E_i\}_{i=1}^\infty \in \mathcal{D}_c$ be a sequence of outward-minimizing sets
such that $E_i \rightarrow E$ in $L^1$ as $i\to\infty$. Then $E\in\mathcal{D}_c$ is also an outward-minimizing set. 
Moreover, $P(E_i) \rightarrow P(E)$ and $\mathcal{H}^{n-1}(\partial_* E_i) \rightarrow \mathcal{H}^{n-1}(\partial_* E)$
as $i\to\infty$.
\end{lemma}

\begin{proof}
Let $F \supset E$. Then by \cite[Proposition 3.38(d)]{afp} and the outward-minimality of $E_i$ we have 
$$P(E_i \cap F) \le P(F)+P(E_i)-P(E_i \cup F) \le P(F).$$
This implies
$$P(E)=P(E\cap F)\le\liminf_i P(E_i\cap F)\le F(F).$$
Hence $E$ is outward-minimizing. By Lemma \ref{neijin}  and Remark \ref{dens}, $E \in \mathcal{D}_c \cap \mathcal{D}^c$.  
It follows from Proposition \ref{outside} that for any $\epsilon>0$, there exists a smooth open set
$O_\epsilon\Supset E$ such that  
$$P(O_\epsilon) \le P(E)+\epsilon.$$
Applying Lemma \ref{neijin} (iii), we have that there exists a sufficiently large $i_0\ge 1$
such that 
$$E_i \subset O_\epsilon, \forall i\ge i_0.$$
This, combined with the outward minimality of $E_i$, implies
$$P(E_i)\le P(O_\epsilon)\le P(E)+\epsilon, \ \forall i\ge i_0.$$
Thus
$$\limsup_i P(E_i) \le P(E).$$
On the other hand, by lower semicontinuity we have 
$$P(E) \le \liminf_i P(E_i).$$
Therefore $P(E_i) \rightarrow P(E)$ as $i\to \infty$.

Since $E_i, E \in \mathcal{D}_c \cap \mathcal{D}^c\ $, 
the last statement follows from Theorem \ref{Federer}.

\end{proof}

Now we are ready to state the main theorem of this section.

\begin{theorem}
\label{bigexistence} The following statements hold:
\begin{itemize}
\item [i)] For $M\ge 1$, the infimum of Problem C in the class of $M$-uniform domains of finite perimeter is attained.
\item [ii)] For $M>1$, the infimum of Problems A, B, C can be attained in the class of
$M$-uniform outer minimal domains.
\end{itemize}
\end{theorem}



\begin{proof}
We first prove i). For a minimizing sequence $(\Omega_i,u_i)$,
where $\Omega_i$ are $M$-uniform domains with finite perimeter and $u_i\in H^1(\Omega_i,\mathbb{S}^2)$.
Let $\hat{u}_i\in H^1(B_{R_0},\mathbb R^3)$ 
be an extension of $u_i$ such that 
$$\|\hat{u}_i\|_{H^1(B_{R_0})} \le C(n,M) \|u_i\|_{H^1(\Omega_i)}.$$
Hence there is a $\hat{u} \in H^1(B_{R_0},\mathbb{R}^3)$ such that
$$\hat{u}_i\rightharpoonup \hat{u} \ {\rm{in}}\ H^1(B_{R_0}).$$
By Theorem \ref{maincompact}, 
there is an $M$-uniform domain $\Omega\subset B_{R_0}$ 
such that $\Omega_i \rightarrow \Omega$ in $L^1$. 
Since $\nabla \hat{u}_i \rightharpoonup \nabla \hat{u}$ in $L^2(B_{R_0})$ and $\chi_{\Omega_i} \rightarrow \chi_{\Omega}$ in $L^1(B_{R_0})$, by the lower semicontinuity we have that
\begin{eqnarray}
\label{kalehaojiu}
\int_{\Omega} |\nabla \hat{u}|^2 \le \liminf_{i \rightarrow \infty} \int_{\Omega_i} |\nabla \hat{u}_i|^2
=\liminf_{i \rightarrow \infty} \int_{\Omega_i} |\nabla {u}_i|^2.
\end{eqnarray}
Denote $u=\hat{u}\big|_{\Omega}$. Then it is not hard to see $|u|=1$ for a.e. $x\in\Omega$
so that $u\in H^1(\Omega,\mathbb S^2)$.
In order to show $(\Omega, u)$ is a minimizer of Problem (C) among $M$-uniform domains of finite perimeter, we have to verify that $u^*=\nu_\Omega$  for $\mathcal{H}^{n-1}$-a.e. on $\partial^* \Omega$. In fact, it follows from $\chi_{\Omega_i}\rightarrow\chi_\Omega$ in $L^2(B_{R_0})$
and $div(\hat{u}_i)\rightharpoonup div(\hat{u})$ in $L^2(B_{R_0})$ and Theorem \ref{traceextensiondomain} that
\begin{eqnarray*}
P(\Omega_i)=\int_{\Omega_i} div ({u}_i) =\int_{B_{R_0}} \chi_{\Omega_i} div (\hat{u}_i) 
&\rightarrow& \int_{B_{R_0}} \chi_\Omega div (\hat{u})=\int_\Omega div(u)\\
&=&\int_{\partial^*\Omega}u^*\cdot \nu_{\Omega}\, d\mathcal{H}^{n-1}\le P(\Omega).
\end{eqnarray*}
This, combined with the lower semicontinuity property of perimeter, implies that $u^*=\nu_\Omega$  for $\mathcal{H}^{n-1}$-a.e. on $\partial^* \Omega$.
Hence the proof of i) is complete.

\medskip
Next, we prove ii). For Problem A in part ii),
let $(\Omega_h,u_h)$ be a minimizing sequence among $M$-uniform, outer minimal domains and $H^1$-unit vector fields on $\Omega_h$. Since $\Omega_h$ are outward-minimizing sets in $B_{R_0}$, 
$P(\Omega_h)$ are uniformly bounded.
By Lemma \ref{heng} and Theorem \ref{maincompact}, we may assume that there exists an $M$-uniform, outer minimal domain $\Omega$ such that up to a subsequence,
 $\Omega_h \rightarrow \Omega$ in $L^1$ and $P(\Omega_h) \rightarrow P(\Omega)$.
As in the proof of i) above, we may extend $u_h$ in $B_{R_0}$, still denoted as $u_h$,
so that $u_h \rightharpoonup u$ in $H^1(B_{R_0}, \mathbb R^3)$ for some $u\in H^1(B_{R_0},\mathbb R^3)$. Thus
we have 
$$\int_{\Omega} |\nabla u|^2 \le \liminf_h \int_{\Omega_h} |\nabla u_h|^2,$$
and $u(x)\in \mathbb S^2$ for a.e. $x\in\Omega$.
 
Since $f$ is  convex, we can write 
$$f(x)=\sup_i(a_ix+b_i).$$
In the following, we do not distinguish $u$ with $u^*$ on $\partial^* \Omega$, and we do not distinguish $\partial^* \Omega_h, \partial^* \Omega$ with $\partial \Omega_h, \partial \Omega$ due to Remark \ref{bianjiea}. Define $$\tau_h(A):=\mathcal{H}^{n-1}(\partial^* \Omega_h \cap A),\ 
\tau(A):=\mathcal{H}^{n-1}(\partial^* \Omega \cap A), \  {\rm{and}}\
\mu_h(A):=\int_{A}f(u_h \cdot \nu_h)d\tau_h,$$
for any measurable $A\subset\mathbb R^n$, where $\nu_h$ is the measure theoretical outer unit normal of $\Omega_h$.  
Then Lemma \ref{heng} implies that 
\begin{eqnarray}
\label{0}
\tau_h(A) \rightarrow \tau(A) \quad \mbox{as $h \rightarrow \infty$}.
\end{eqnarray}
Since $f$ is bounded and nonnegative, $\mu_h$ are nonnegative Radon measures so that 
we may assume there is a nonnegative Radon measure $\mu$ such that after passing to a subsequence, 
$\mu_h \rightharpoonup \mu$ as $h\to\infty$ as weak convergence of Radon measures. 
Decompose $\mu$ as $\mu=(D_{\tau}\mu)\tau+\mu^s, \mu^s \perp \tau$, and $\mu^s\ge 0$. Then 
\begin{eqnarray}
\liminf_{h \rightarrow \infty} \mu_h(A) \ge \mu(A) \ge \int_A D_{\tau} \mu d\tau.
\end{eqnarray} 
It follows from Theorem \ref{Federer} that $x \in \partial^* \Omega$ holds for $\tau$-a.e. $x \in B_{R_0}$. 
Now any such $x\in\partial^*\Omega$, we claim that there exists $r_j\to 0$ such that for  $B_j=B_{r_j}(x)$, it holds that
\begin{itemize}
\item[(a)] $\mathcal{H}^{n-1}(\partial B_j \cap \partial \Omega)=0$ and $\mathcal{H}^{n-1}(\partial B_j \cap \partial \Omega_h)=0, \, \forall h\ge 1$.
\item[(b)] $\displaystyle\int_{\partial B_j \cap \Omega_h} u_h \cdot \nu_{B_j}\,d\mathcal{H}^{n-1} \rightarrow \int_{\partial B_j \cap \Omega} u \cdot \nu_{B_j}\,d\mathcal{H}^{n-1}$\ as\ $h \rightarrow \infty$.
\item[(c)] $\mu(\partial B_j)=0$.
\item[(d)] $D_{\tau} \mu(x)=\displaystyle\lim_j\frac{\mu(B_j)}{\tau(B_j)}$ and $\displaystyle\lim_{j\rightarrow \infty} \frac{\int_{B_j}u \cdot \nu_{B_j} d\tau}{\tau(B_j)}=u(x)\cdot \nu(x)$.
\end{itemize}
Indeed, (a) and (c) are true because $\tau, \tau_h,$ and $\mu$ are nonnegative Radon measures. 
(d) follows from the Lebesgue differentiation Theorem. To see (b), let $\tilde{u}_h=u_h \chi_{\Omega_h}$ and $\tilde{u}=u\chi_{\Omega}$. Since $\tilde{u}_h \rightarrow \tilde{u}$ in $L^1$, we have 
\begin{eqnarray*}
\int_{B_1(x)} |\tilde{u_h}-\tilde{u}| = \int_0^1 \int_{\partial B_r(x)} |\tilde{u}_h-\tilde{u}| d\mathcal{H}^{n-1} dr \rightarrow 0 \quad \mbox{as $h \rightarrow \infty$}.
\end{eqnarray*}Therefore by Fatou's Lemma, 
$$\int_0^1 \liminf_{h \rightarrow \infty}\int_{\partial B_r(x)}|\tilde{u}_h-\tilde{u}|d\mathcal{H}^{n-1}\,dr=0,$$ 
hence for almost every $r \in (0,1)$ and for a subsequence of $h \rightarrow \infty$, 
\begin{eqnarray*}
&&\big|\int_{\partial B_r(x) \cap \Omega_h} u_h \cdot \nu_{B_r(x)}\,d\mathcal{H}^{n-1}-\int_{\partial B_r(x) \cap \Omega} u \cdot \nu_{B_r(x)}\,d\mathcal{H}^{n-1}\big|\\
&& \le  \int_{\partial B_r(x)} |\tilde{u}_h-\tilde{u}| d\mathcal{H}^{n-1} \rightarrow 0.
\end{eqnarray*} 
This finishes the proof of (b).
Now we return to the proof of v). By (c), 
\begin{eqnarray*}
\label{3}
\mu(B_j) = \lim_{h \rightarrow \infty} \mu_h(B_j) = \lim_{h \rightarrow \infty} \int_{\partial \Omega_h \cap B_j} f(u_h\cdot \nu_h)\,d\mathcal{H}^{n-1}.
\end{eqnarray*}
Also as $h \rightarrow \infty$, up to a subsequence we have
\begin{eqnarray*}
&&\int_{\partial \Omega_h \cap B_j} u_h \cdot \nu_h\,d\mathcal{H}^{n-1}\\
&&= \int_{\partial(\Omega_h \cap B_j)} u_h \cdot \nu_{\Omega_h \cap B_j}\,d\mathcal{H}^{n-1} - \int_{\partial B_j \cap \Omega_h} u_h \cdot \nu_{B_j} \,d\mathcal{H}^{n-1}, \\
&&= \int_{ \Omega_h \cap B_j} divu_h-\int_{\partial  B_j \cap \Omega_h} u_h \cdot \nu_{B_j} \,d\mathcal{H}^{n-1},\\ 
&&\rightarrow\int_{\Omega \cap B_j} divu - \int_{\partial B_j \cap \Omega} u \cdot \nu_{B_j} \,d\mathcal{H}^{n-1}, \\
&&= \int_{\partial(\Omega \cap B_j)} u \cdot \nu_{\Omega \cap B_j}\,d\mathcal{H}^{n-1} - \int_{\partial B_j \cap \Omega} u \cdot\nu_{B_j} \,d\mathcal{H}^{n-1}\\
\label{5.11}
&&=\int_{\partial \Omega \cap B_j} u \cdot \nu_\Omega\,d\mathcal{H}^{n-1}.
\end{eqnarray*}
Therefore, for $\tau$-a.e. $x\in B_{R_0}$, it follows
\begin{eqnarray}
D_{\tau} \mu(x)&=&\lim_j\frac{\mu(B_j)}{\tau(B_j)}\\
&=&\lim_j\lim_h \frac{ \int_{\partial \Omega_h \cap B_j} f(u_h\cdot \nu_h)\,d\mathcal{H}^{n-1}}{\mathcal{H}^{n-1}(\partial \Omega \cap B_j)}\nonumber\\
& \ge &\lim_j\lim_h \frac{ \int_{\partial \Omega_h \cap B_j}(a_i u_h\cdot \nu_h+b_i)\,d\mathcal{H}^{n-1}}{\mathcal{H}^{n-1}(\partial \Omega \cap B_j)}\nonumber\\
&=&\lim_j\frac{ \int_{\partial \Omega \cap B_j}(a_i u\cdot \nu_\Omega+b_i)\,d\mathcal{H}^{n-1}}{\mathcal{H}^{n-1}(\partial \Omega \cap B_j)},\quad \mbox{also by}\, \eqref{0}\nonumber\\
&=& a_i u(x) \cdot \nu_\Omega(x) +b_i. \nonumber
\end{eqnarray}
Hence $D_{\tau}\mu \ge f(u\cdot \nu_\Omega)$ for $\tau$- a.e. $x\in B_{R_0}$, and 
\begin{eqnarray}
\liminf_h \int_{\partial \Omega_h} f(u_h \cdot \nu_h)\,d\mathcal{H}^{n-1}&=&\liminf_h \mu_h(B_{R_0}) \ge \int_{B_R} D_{\tau} \mu d\tau\nonumber\\
&\ge& \int_{B_{R_0}} f(u\cdot \nu) d\tau =\int_{\partial \Omega} f(u\cdot\nu)\,d\mathcal{H}^{n-1}.
\end{eqnarray}
Therefore, $(\Omega,u)$ is a minimizer.\\

To complete the proof of statements in ii), it remains to show if $(\Omega_i, u_i)$ are a minimizing sequence in Problem (B) and converges weakly to $(\Omega,u)$, 
then $u \cdot \nu = c$ for $\mathcal{H}^{n-1}$-a.e. on $\partial^* \Omega$. This can be seen from
$$\liminf_{i \rightarrow \infty} \int_{\partial^* \Omega_i} f(u_i\cdot \nu_i)d\mathcal{H}^{n-1}\ge \int_{\partial^* \Omega}f(u \cdot \nu)d\mathcal{H}^{n-1}.$$  
In fact, by choosing $f(t)=\mu (t-c)^2$ we have that
\begin{eqnarray}
\label{614}
\int_{\partial^* \Omega} (u \cdot \nu-c)^2\,d\mathcal{H}^{n-1}\le \liminf_{i \rightarrow \infty} \int_{\partial^* \Omega_i} (u_i \cdot \nu_i-c)^2\,d\mathcal{H}^{n-1}=0.
\end{eqnarray} 
Hence $u \cdot \nu\equiv c$ for $\mathcal{H}^{n-1}$-a.e. on $\partial^* \Omega$. This completes the proof.
\end{proof}

\section{On the uniqueness of Problem C}
In this section, we will show the uniqueness of Problem C in the class of $C^{1,1}$-star-shaped, mean convex domains in $\mathbb R^3$.
We will assume the domains has volume $V_0=|B_1|$, where $B_1\subset\mathbb R^3$ is the unit ball centered at $0$. We begin with




\begin{lemma} For any bounded $C^{1,1}$-domain $\Omega\subset\mathbb R^3$, 
\label{2}
\begin{equation}
\label{2.1}
\inf\big\{\int_{\Omega}|\nabla u|^2\ \big|\  u \in H^1(\Omega, \mathbb S^2), u=\nu_\Omega \, \mbox{on $ \partial \Omega$}\big\} \ge \int_{\partial \Omega} H_{\partial\Omega}\,d\mathcal{H}^2,
\end{equation} where $H_{\partial\Omega}$ is the mean curvature of $\partial\Omega$.
\end{lemma}

\begin{proof} Let $u\in H^1(\Omega, \mathbb S^2)$, with $u=\nu_\Omega$ on $\partial\Omega$, be such that 
$$\int_\Omega |\nabla u|^2=\inf\big\{\int_{\Omega}|\nabla u|^2\ \big|\  u \in H^1(\Omega, \mathbb S^2), u=\nu_\Omega \, \mbox{on $ \partial \Omega$}\big\}.$$
Then by \cite{SU1, SU2}, $u\in C^\infty(\Omega\setminus\{a_i\}_{i=1}^N,\mathbb S^2)$ for a finite set $\displaystyle\cup_{i=1}^N \{a_i\}\Subset\Omega$. 
Observe that
$$(div(u))^2-tr(\nabla u)^2=div(div(u)u-(\nabla u) u)\ \ \ {\rm{in}}\ \ \Omega\setminus\cup_{i=1}^N \{a_i\}.$$
By \cite[Proposition 2.2.1]{LW}, we have that
$$|\nabla u|^2 \ge (divu)^2-tr(\nabla u)^2 \ \ \ {\rm{in}}\ \ \Omega\setminus\cup_{i=1}^N \{a_i\}.$$
By \cite[Theorem 1.9]{AL}, near each $a_i$, $u(x) \sim R(\frac{x-a_i}{|x-a_i|})$ for some rotation $R\in O(3)$. In particular, one has that for $r>0$ sufficiently small,
$$\Big|\int_{\partial B_r(a_i)} (div(u)u-(\nabla u) u)\cdot \nu_{B_r(a_i)}\,d\mathcal{H}^2\Big| =O(r).$$ 
Hence
\begin{eqnarray*}
&&\int_{\Omega} |\nabla u|^2\\
&&\ge \int_{\Omega \setminus \cup_{i=1}^NB_r(a_i)} (div(u))^2-tr(\nabla u)^2 \\
                           &&= \int_{\Omega \setminus \cup_{i=1}^NB_r(a_i)} div((divu)u-(\nabla u) u) \\
                           &&= \int_{\partial \Omega} (div(u)u-(\nabla u) u)\cdot \nu_{\Omega}\,d\mathcal{H}^2\\
                           &&\ - \sum_{i=1}^n\int_{\partial B_r(a_i)} (div(u)u-(\nabla u) u)\cdot \nu_{B_r(a_i)}\,d\mathcal{H}^2\\
                           &&\ge \int_{\partial \Omega} \big(div (u)-((\nabla u)\nu_\Omega) \cdot \nu_\Omega \big)\,d\mathcal{H}^2-CN r \\
                           &&= \int_{\partial \Omega} \big(div_{\partial \Omega} \nu_\Omega\big)\,d\mathcal{H}^2 -CN r
                           =\int_{\partial \Omega} H_{\partial\Omega}\,d\mathcal{H}^2-CN r.
\end{eqnarray*}
This implies \eqref{2.1} after sending $r \rightarrow 0$. 
\end{proof}

The inequality \eqref{2.1} leads us to study the minimization of the total mean curvatures. It is well-known that 
\begin{equation}
\label{key}
\int_{\partial \Omega} H_{\partial\Omega}\,d\mathcal{H}^2 \ge 4\sqrt{\pi P(\Omega)} 
\end{equation}
is true if $\Omega$ is convex, and the equality holds if and only if $\Omega$ is a ball. 
Very recently, Dalphin-Henrot-Masnou-Takahashi \cite{DHMT} proved that if $\Omega$ is a revolutionary solid and $H \ge 0$, then \eqref{key} is true,
 and the equality holds if and only if $\Omega$ is a ball. Without the mean convexity, \eqref{key} is false, see \cite{DHMT}. In the next lemma we present a proof that \eqref{key} is true if $\Omega$ is a $C^{1,1}$ star-shaped and mean convex domain. The key ingredient of the proof is based on the result by Gerhardt \cite{G90}. We remark that a more general version of \eqref{key} has been proven by Guan-Li \cite{GL}. 
Here we will sketch the proof, since it is elementary in $\mathbb R^3$.

\begin{lemma}
\label{3}
The inequality \eqref{key} holds, if $\Omega$ is $C^{1,1}$-strictly star-shaped and mean convex. 
\end{lemma}

\begin{proof}
By the remark below, we may assume $\Omega\in C^\infty$. By a standard argument, we can perturb $\Omega$ so that $H>0$ everywhere. 
Indeed, represent $\partial \Omega$ as an embedding $F^0:\mathbb S^2 \to \mathbb R^3$ and consider the mean curvature flow $\{F_t :\mathbb S^2 \to \mathbb R^3: t\in [0,T)\}$, which is a family of embeddings so that 
$$\frac{\partial F}{\partial t} = H \nu_t \ \ 0<t<T; \ \ F_0=F^0, $$
where $\nu_t$ is the inward unit normal of the embedding $F_t$. It is well-known that the solution exists for a short time $T>0$. 
If $t>0$ is small, then $F_t(\mathbb S^2)$ remains to be star-shaped. The evolution of the mean curvature $H$ of $F_t(\mathbb S^2)$ is given by 
$$\frac{\partial H}{\partial t} = \Delta H + |A|^2 H,$$
where $A$ is the second fundamental form of $F_t(\mathbb S^2)$. 
Then the strong maximum principle implies that $H>0$ everywhere on $F_t(\mathbb S^2)$ for $t>0$. 
It is clear that after a small perturbation in $C^1$-norm, $\Omega$ is still strictly star-shaped.

Hence it suffices to prove \eqref{key} by assuming $H>0$ everywhere on $\partial \Omega$. 
We argue it by contradiction. Suppose there were a strictly star-shaped domain $\Omega$ with $H>0$ everywhere on $\partial \Omega$ such that 
$$\frac{\int_{\partial \Omega} H \,d\mathcal{H}^2}{ 4\sqrt{\pi P(\Omega)}}<1.$$
Representing $\partial \Omega$ as an embedding $G_0:\mathbb S^2 \to \mathbb R^3$. Now consider the inverse mean curvature flow $\{G_t :\mathbb S^2 \to \mathbb R^3: t\in [0,\infty)\}$, 
which is a family of embeddings that solves
$$\frac{\partial G}{\partial t} = \frac{1}{H} \nu_t,$$
where $\nu_t$ is the inward unit normal of the embedding $G_t$. It has been shown by Gerhardt \cite{G90} that $S_t:=G_t(\partial \Omega)$ converges to 
the unit sphere $\mathbb S^2$, up to rescalings by $e^{-t/2}$, as $t\to \infty$. Set 
$$y(t)=\frac{\int_{S_t} H\,d\mathcal{H}^2} { 4\sqrt{\pi Area(S_t)}}, \ t>0.$$
Observe that $y(t)$ is scaling-invariant. Therefore, $y(0)<1$ and $y(t) \rightarrow 1$ as $t\to\infty$. 
On the other hand, using the evolution equations under the inverse mean curvature flow we have that 
$$\frac{d}{dt}H=-\Delta H-\frac{|A|^2}{H},$$ and $$\frac{d}{dt}\sqrt{g}=\sqrt{g}, $$ where $\Delta$ is the surface Laplacian and $g$ is the metric on surface $S_t$ induced by Euclidean metric
in $\mathbb R^3$. Direct calculations imply 
\begin{eqnarray*}
\frac{d}{dt}\left(\frac{\int_{S_t}H\,d\mathcal{H}^2}{4\sqrt{\pi P(\Omega)}}\right)&=& \Big(\int_{S_t}\big(H-\frac{|A|^2}{H}\big)\,d\mathcal{H}^2\Big)\frac{1}{4\sqrt{\pi Area(S_t)}}-\frac{\int_{S_t}H\,d\mathcal{H}^2}{8\sqrt{\pi Area(S_t)}}\\
&=& \frac{1}{4\sqrt{\pi Area(S_t)}}\left(\int_{S_t}\frac{2K}{H}\,d\mathcal{H}^2-\frac{1}{2}\int_{S_t}H\,d\mathcal{H}^2\right)\\
&=& \frac{1}{4\sqrt{\pi Area(S_t)}}\int_{S_t} \frac{4K-H^2}{2H}\, d\mathcal{H}^2 \le 0,
\end{eqnarray*}
since $H^2 \ge 4K$, here $K$ is the Gauss curvature of $S_t$. 
Therefore,  $y(t) \le y(0)<1$ for all $t>0$. We get a desired contradiction.
\end{proof}


\begin{remark}
\label{hsa}
\eqref{key} is actually  true for any $C^1$-strictly star-shaped surface with bounded nonnegative generalized mean curvature, in particular for a $C^{1,1}$-mean convex surface.
Indeed, by \cite[Lemma 2.6]{HI2}, we can find a family of smooth strictly star-shaped mean convex hypersurfaces converging to the surface uniformly in $C^{1,\alpha} \cap W^{2,p}$ for $0<\alpha<1$ and $1<p<\infty$ so that the total mean curvature of the smooth surfaces converges to the total mean curvature of the original surface. 
We refer the reader to \cite{HI2} for the detail. 
\end{remark}

By Lemma \ref{3} and the isoperimetric inequality $P(\Omega) \ge 4\pi(\frac{3}{4 \pi}|\Omega|)^{2/3}$,  we immediately have

\begin{corollary} 
\label{4}It holds that
\begin{eqnarray*}
&&\inf\{\int_{\Omega}|\nabla u|^2: \mbox{$\Omega$ is $C^{1,1}$-star-shaped, mean convex}, |\Omega|=|B_1|, u \in H^1(\Omega, \mathbb S^2), \\
&&\qquad\qquad u=\nu_\Omega \, \mbox{on $ \partial \Omega$}\}\ge 8\pi,
\end{eqnarray*}
and the equality holds if and only if $\Omega=B_1$, up to translation and rotation.
\end{corollary}

As a consequence, we have 

\begin{theorem}
\label{dafeiji1}
The Problem (C) over $C^{1,1}$-star-shaped and mean convex domains is uniquely achieved at $\Omega=B_1$ and $u(x)=\frac{x}{|x|}$.
\end{theorem}

\begin{proof}
By direct calculations, 
$$\int_{B_1} |\nabla (\frac{x}{|x|})|^2 =\int_{B_1} \frac{2}{|x|^2}=8\pi.$$ 
Hence by the first statement in Corollary \ref{4}, \eqref{0} is attained at $(B_1, \frac{x}{|x|})$. 
The uniqueness follows from the last statement of Corollary \ref{4} and \cite[Theorem 7.1]{BCL}.
\end{proof}

\begin{remark}
\label{dafeiji2}
Huisken first proves that \eqref{key} holds if $\Omega$ is $C^{1,1}$-outer minimal (not necessarily connected), though it seems that he didn't publish it. See also 
Freire-Schwartz \cite[Theorem 5]{FS}. Hence the same result as in Theorem \ref{dafeiji1} holds in the  class of $C^{1,1}$-outer minimal open sets. 
By \cite{DHMT}, the same result as in Theorem \ref{dafeiji1} holds in the class of smooth domains of revolution. 
\end{remark}

\bibliographystyle{amsalpha}

\begin{thebibliography}{A}



\bibitem{Adams} D. Adams and L. Hedberg, \newblock{Function spaces and potential theoery}.  Grundlehren der mathematischen Wissenschaften, Volume 314, 1996.

\bibitem{afp} L. Ambrosio, N. Fusco, and D. Pallara, \newblock{Functions of bounded variation and free discontinuity problems}. Oxford Mathematical Monographs. Clarendon Press, Oxford University Press,
New York, 2000.

\bibitem{AL} F. Almgren and E. Lieb, \newblock{\it Singularities of energy minimizing maps from the ball
to the sphere: examples, counterexamples, and bounds}. \newblock{Ann. of Math.} (2) 128
(1988), no. 3, 483-530.

\bibitem{BGM} E. Barozzi, E. Gonzalez and U. Massari, \newblock{\it Pseudoconvex sets}. \newblock{Ann.Univ.Ferrara.} 55 (2009), 23-35.

\bibitem{BH} F. Bayart and Y. Heurteaux, \newblock{\it On the Hausdorff Dimension of Graphs
of Prevalent Continuous Functions on Compact Sets}. \newblock{Real Anal. Exchange},
37 (2011), no. 2, 333-352.

\bibitem{Beer} G. Beer, \newblock{\it The Hausdorff metric and convergence in measure}. 
\newblock{Michigan Math. J}. 21 (1974), no. 1,  63-64.

\bibitem{Beer2} G. Beer,
\newblock {\it Starshaped sets and the Hausdorff metric.}
\newblock{Pacific J. Math}, 61 (1975), no. 1, 21-27.

\bibitem{BCL} H. Brezis, J. Coron and E. Lieb, \newblock{\it Harmonic maps with defects.}
\newblock{Comm. Math. Phys.} 107 (1986), no. 4,  649-705.

\bibitem{CTZ} G. Chen, M. Torres and W. Ziemer, \newblock{\it Gauss-Green theorem for weakly differentiable vector fields, sets of finite perimeter, and balance laws}. 
\newblock{Comm. Pure Appl. Math.}, 62 (2009), no 2, 242-304.

\bibitem{DHMT} J. Dalphin, A. Henrot, S. Masnou and T.Takahashi,
\newblock{\it On the minimization of total mean curvature.} \newblock{J. Geometric Anal.}, 26 (2016), no. 4, 2729-2750.

\bibitem{DLW} H. Du, Q. Li and C. Wang, \newblock{\it Compactness of $M$-uniform domains and optimal thermal insulation problems.} \newblock{Adv. Cal. Var.}, in press. 

\bibitem{Ev} L. Evans, \newblock{Weak convergence methods for nonlinear partial differential equations}.
\newblock{CBMS Regional Conference Series in Mathematics, 74. Published for the Conference Board of the Mathematical Sciences, Washington, DC; by the American Mathematical Society, Providence, RI, 1990. viii+80 pp.}

\bibitem{EG} L. Evans and R. Gariepy, \newblock{Measure theory and fine properties of functions}.
\newblock{Studies in Advanced Mathematics. CRC Press, Boca Raton, FL, 1992. viii+268 pp.}

\bibitem{Fa}K. J. Falconer, \newblock{The geometry of fractal sets}. \newblock{Cambridge Tracts in Mathematics}, 85. Cambridge University Press, Cambridge, 1986.

\bibitem{FE} H. Federer, \newblock{\it The area of a nonparametric surface}. \newblock{Proc. Amer. Math. Soc.}, 11 (1960), 436-439.

\bibitem{FS} A. Freire and F. Schwartz, \newblock{\it Mass-capacity inequalities for conformally flat manifolds with boundary}. \newblock{Comm. PDE}, 39 (2014), 98-119.

\bibitem{FF} A. Ferriero and N. Fusco, \newblock{\it A note on the convex hull of sets of finite perimeter in the plane}. \newblock{Discrete Cont. Dyn. Syst., Series B}, 
11 (2009), no. 1, 102-108.

\bibitem{FM}A. Figalli and F. Maggi, \newblock{\it On the Shape of Liquid Drops and Crystals
in the Small Mass Regime}, \newblock{Arch. Rational Mech. Anal.}, 201 (2011), 143-207.


\bibitem{GL} Z. Y. Geng, F. H. Lin, 
\newblock{\it The two-dimensional liquid crystal droplet problem with a tangential boundary condition},
\newblock{Arch. Ration. Mech. Anal.},  243 (2022), no. 3, 1181-1221.


\bibitem{GO} F. W. Gehring and B. G. Osgood, \newblock{\it Uniform domains and the quasihyperbolic metric.}
\newblock{J. Analyse Math.} 36 (1979), 50-74 (1980).


\bibitem{G90} C. Gerhardt, \newblock{\it Flow of nonconvex surfaces into spheres.} 
\newblock {J. Differential Geom.}, 32 (1990), no. 1, 299-314.

\bibitem{Giusti} E. Giusti, \newblock{\it Minimal Surfaces and Functions of Bounded Variation}.
\newblock{Monographs in Mathematics}, Volume 80, 1984.


\bibitem{GL} P. F. Guan and J. Y.  Li, \newblock{\it The quermassintegral inequalities for k-convex starshaped domains.} {Adv. Math.}, 221 (2009), no. 5, 1725?1732


\bibitem{GHL} C. F.  Gui, Y. Y. Hu and Q. F. Li. \newblock{\it On smooth interior approximation of Sets of Finite Perimeter.}
{Proc. Amer. Math. Soc.}, to appear, arXiv:2210.11734.


\bibitem{Har} P. Harjulehto, \newblock{\it Traces and Sobolev extension domains}.
{Proc. Amer. Math. Soc.}, 134 (2006), no. 8, 2373-2382.

\bibitem{HI1} G.  Huisken and T. Ilmanen, \newblock{\it The Inverse Mean Curvature Flow and the Riemannian Penrose Inequality}. {J. Differential Geom.}  
59 (2001), no. 3, 353?437.

\bibitem{HI2} G. Huisken and T. Ilmanen, \newblock{\it Higher regularity of the inverse mean curvature flow}.
{J. Differential Geom.} \textbf{80} (2008), no. 3, 433-451.

\bibitem{HU} B. Hunt, \newblock{\it The Hausdorff dimension of graphs of Weirestrass functions}.
\newblock{Proc. Amer. Math. Soc.}, 126 (1998), no. 3, 791-800.

\bibitem{Jo} P. W. Jones, \newblock{\it Quasiconformal mappings and extendability of functions in Sobolev spaces}. \newblock{Acta Math.}, 147 (1981), no.1-2,  71-88.

\bibitem{KVW} D. Kalaj, M. Vuorinen and G. Wang, \newblock{\it On quasi-inversions}.
\newblock{Monatsh Math}, 180 (2016), no. 4, 785-813. 

\bibitem{LV} R. Lachiasze-Rey and S. Vega, \newblock{\it Boundary density and Voronoi set estimation for irregular sets}. \newblock{Trans. Amer. Math. Soc.}, 369 (2017), no. 7, 4953-4976. 

\bibitem{Li} Q. F. Li, \newblock{Geometric Measure Theory with Applications to Shape Optimization Problems.}
Thesis (Ph.D.)-Purdue University. 2018. 249 pp. ISBN: 978-0438-01844-0, ProQuest LLC.

\bibitem{LT} Q. F. Li and M. Torres, \newblock{\it Morrey spaces and generalized Cheeger set}. 
\newblock{Adv. Cal. Var.}, 12 (2019) no. 2, 111-133.

\bibitem{LP} F. H. Lin and C. C. Poon, \newblock{\it On nematic liquid crystal droplets}.
\newblock{Elliptic and Parabolic Methods in Geometry.} (Minneapolis, MN, 1994), 91-121, A K Peters, Wellesley, MA, 1996. 

\bibitem{LW} F. H.  Lin and C. Y.  Wang,
\newblock {The analysis of harmonic maps and their heat flows.}
\newblock {World Scientific Publishing Co. Pte. Ltd., Hackensack, NJ, 2008. xii+267 pp.}

\bibitem{Maggi} F. Maggi,
\newblock  Sets of finite perimeter and geometric variational problems.
 \newblock An introduction to geometric measure theory. Cambridge Studies in Advanced Mathematics, 135. Cambridge University Press, Cambridge, 2012.
 
\bibitem{MS} O. Martio and U. Srebro, \newblock{\it On the existence of automorphic quasimeromorphic mappings in $\mathbb{R}^n$}. 
\newblock{Ann. Acad. Sci. Fenn. Ser. A I Math}. 3 (1977), no. 1, 123-130.


\bibitem{PWZ} P. Sternberg, G. Williams and W.P. Ziemer, \newblock{\it $C^{1,1}$-Regularity of Constrained Area
Minimizing Hypersurfaces}. \newblock{J. Differential Equations} 94 (1991), no. 1, 83-94.

\bibitem{St} E. M. Stein, \newblock{Singular integrals and differentiability properties of functions}.
Princeton, University Press, Princeton. N.J.1970.

\bibitem{SU1} R. Schoen, K. Uhlenbeck, {\it A regularity theory for harmonic maps}.{J. Differential Geom.} 17 (1982), no. 2, 307-335.

\bibitem{SU2} R. Schoen, K. Uhlenbeck, {\it Boundary regularity and the Dirichlet problem for harmonic maps.}
{J. Differential Geom.} 18 (1983), no. 2, 253-268.

\bibitem{Ta} I. Tamanini, \newblock{\it Boundaries of Caccioppoli Sets with H\"older-continuous normal
cector}. \newblock{J. Reine Angew. Math.}, 334 (1982) 27-39.



\end{thebibliography}

\end{document}